\documentclass[11pt]{amsart}
\usepackage{enumitem, amssymb}

\usepackage{fontawesome}
\usepackage{hyperref}
\usepackage{tikz}
\numberwithin{equation}{section}

\DeclareMathOperator{\SL}{SL}

\DeclareMathOperator{\Aut}{Aut}

\newcounter{my_enumerate_counter}

\newtheorem{thm}{Theorem}
\newcommand{\bbC}{{\mathbb C}}

\newcommand{\bbN}{{\mathbb N}}

\newcommand{\calL}{{\mathcal L}}

\newcommand{\cU}{{\mathcal U}}

\newcommand{\tr}{{\rm tr}}

\newcommand{\Hom}{{\rm Hom}}

\newcommand{\wstar}{$\mathrm{W}^*$}

\newcommand{\cstar}{$\mathrm{C}^*$}

\newcommand{\cstr}{\mathrm{C}^*_r}

\DeclareMathOperator{\id}{id}
\newcommand{\cO}{\mathcal O}

\newcommand{\bt}{\mathsf t}

\newcommand{\bbZ}{{\mathbb Z}}

\newcommand{\bbR}{\mathbb R}

\newcommand{\rs}{\restriction}
 
\makeatletter
\newcommand{\customlabel}[2]{%
\protected@write \@auxout {}{\string \newlabel {#1}{{#2}{\thepage}{}{}{}}}}
\makeatother

\newtheorem{theorem}{Theorem}[section]
\newtheorem*{theorem*}{Theorem}
\newtheorem{proposition}[theorem]{Proposition}

\newtheorem*{proposition*}{Proposition}
\newtheorem{lemma}[theorem]{Lemma}
\newtheorem*{lemma*}{Lemma}

\newtheorem*{corollary*}{Corollar}

\newtheorem*{fact*}{Fact}
\theoremstyle{definition}
\newtheorem{definition}[theorem]{Definition}
\newtheorem*{definition*}{Definition}

\newtheorem*{claim*}{Claim}

\newtheorem{conjecture}[theorem]{Conjecture}
\newtheorem*{conjecture*}{Conjecture}

\theoremstyle{remark}

\newtheorem*{example*}{Example}

\newtheorem*{remark*}{Remark}

\newtheorem*{note*}{Note}
\newtheorem*{question*}{Question}

\DeclareMathOperator{\frFx}{\mathfrak F^{\bar x}}
\DeclareMathOperator{\frFxQF}{\mathfrak F_{QF}^{\bar x}}

\DeclareMathOperator{\frFxyQF}{\mathfrak F_{QF}^{\bar x,y}}

\author[Ilijas Farah]{Ilijas Farah}
\thanks{Partially supported by NSERC}
\address{Department of Mathematics and Statistics,
York University,
4700 Keele Street,
Toronto, Ontario, Canada, M3J
1P3} 
\address{Matemati\v cki Institut SANU,
Kneza Mihaila 36,
11\,000 Beograd, p.p. 367,
Serbia}
\email{email: ifarah@yorku.ca}
\urladdr{http://www.math.yorku.ca/$\sim$ifarah}

\thanks{ORCID iD https://orcid.org/0000-0001-7703-6931}
	\dedicatory{To the memory of Eberhard Kirchberg.}

\title{Quantifier elimination in II$_1$ factors}
\date{\today}

\begin{document}

\begin{abstract}
	No type II$_1$ tracial von Neumann algebra has theory that admits quantifier elimination.
\end{abstract}
\maketitle

Model theory is largely the study of definable sets, more precisely sets definable by first order (continuous or discrete) formulas. If the theory of a structure admits elimination of quantifiers, then its definable subsets are definable by quantifier-free formulas and therefore easier to grasp. This is the case with atomless Boolean algebras, dense linear orderings without endpoints, real closed fields, divisible abelian groups\dots See for example \cite{Mark:Model}.
Quantifier elimination is equivalent to an assertion about embeddings between finitely-generated substructures of an ultrapower that resembles the well-known fact that all embeddings of the hyperfinite II$_1$ factor  into its ultrapower are unitarily equivalent (Proposition~\ref{P.QE}, Proposition~\ref{P.QEgeneral}).    
Our motivation for studying quantifier elimination in tracial von Neumann stems from Jekel's work on 1-bounded entropy in \wstar-algebras defined on types (\cite{jekel2022free}, \cite{jekel2022covering}). Our main result implies that 1-bounded entropy is a genuine generalization of Hayes's 1-bounded entropy  (see \cite{hayes2018}, \cite{hayes2021random}). This is because the latter is defined for quantifier-free types, while the former is defined for full types.

In \cite[Theorem~2.1]{GoHaSi:Theory} it was proven that the theory of the  hyperfinite~II$_1$ factor does not admit quantifier elimination (see \S\ref{S.QE}) and that if $N$ is a~II$_1$ McDuff factor such that every separable II$_1$ factor is $N^\cU$-embeddable\footnote{All ultraproducts in this paper are associated to nonprincipal ultrafilters on $\bbN$ and almost all of them are tracial.} then the theory of $N$ does not admit elimination of quantifiers (\cite[Theorem~2.2]{GoHaSi:Theory}). 

In spite of the slow start, the question of quantifier elimination in \cstar-algebras has been answered completely. 
In \cite[Theorem~1.1]{eagle2015quantifier} it was proven that the only \cstar-algebras whose theories in the language of unital \cstar-algebras admit quantifier elimination are (with $K$ denoting the Cantor space) $\bbC$, $\bbC^2$, $M_2(\bbC)$ and $C(K)$, and that the only \cstar-algebras whose theories in the language without a symbol for a unit 
 admit quantifier elimination are $\bbC$ and $C_0(K\setminus \{0\})$. 
 
 The key component of \cite{eagle2015quantifier} is an observation due to Eberhard Kirchberg,\footnote{This is the paragraph with norm ultraproducts.} that there are two very different embeddings of $\cstr(F_2)$ into the ultrapower of the Cuntz algebra $\cO_2$.   First, by \cite{KircPhi:Embedding}, every exact \cstar-algebra embeds into~$\cO_2$ and diagonally into its ultrapower. Second, by~\cite{haagerup2005new}, $\cstr(F_2)$ embeds into $\prod_\cU M_n(\bbC)$, giving a nontrivial embedding into the ultrapower of $\cO_2$. 
 
  Later on,  in  \cite{goldbring2022survey}, paragraph preceding Lemma 5.2, it was pointed out without a proof that the argument from \cite{GoHaSi:Theory} implies that no McDuff II$_1$ factor admits elimination of quantifiers. Also, in   \cite[Proposition~4.17]{goldbring2017theories} it was proven that if $N$ is a II$_1$ factor such that $N$ and $M_2(N)$ have the same universal theory and $N$ is existentially closed for its univeral theory, then $N$ is McDuff. This implies that if $N$ is not McDuff and $M_2(N)$ embeds unitally into~$N^\cU$, then $N$ does not admit elimination of quantifiers. 

During a reading seminar on model theory and free probability based on~\cite{jekel2022free} at York University in the fall semester of 2022, the author (unaware of the recent developments described in the last paragraph) rediscovered an easy argument that the theory of a non-McDuff factor $N$ such that $M_2(N)$ embeds unitally into $N^\cU$, does not admit quantifier elimination  and asserted that closer introspection of the proof of \cite[Theorem~2.1]{GoHaSi:Theory} ought to yield the same result for all II$_1$ factors. This is indeed the case; the following answers the \cite[Question~2.18]{jekel2022free} in case of tracial von Neumann algebras of type II$_1$.

\begin{thm}\label{T.1}
	If $N$ is a tracial von Neumann algebra with a direct summand of type II$_1$, then the theory of $N$ does not admit elimination of quantifiers. 
\end{thm}

The key lemma is proven in \S\ref{S.Nate} and  
the proof of Theorem~\ref{T.1} can be found at the end of this section.  The well-known criterion for quantifier elimination is proven in \S\ref{S.Criterion} (in spite of it being well-known, a self-contained proof of this fact was not available in the literature until it appeared in \cite[Proposition~9.4]{hart2023introduction}). In \S\ref{S.Concluding} we state conjectures on when the theory of a tracial von Neumann algebra is model-completene and when it admits quantifier elimination.

Our terminology is standard. For model theory see \cite{BYBHU}, \cite{hart2023introduction}, for general operator algebras \cite{Black:Operator}, for II$_1$ factors  \cite{anantharaman2017introduction}, and for model theory of tracial von Neumann algebras \cite{jekel2022free} and \cite{goldbring2022survey}.

For simplicity of notation, every tracial state is denoted $\tau$, except those on $M_n(\bbC)$, denoted $\tr_n$. 

A tracial von Neumann algebra is \emph{$R^\cU$-embeddable} if it embeds into some (equivalently, any) ultrapower of the hyperfinite II$_1$ factor $R$ associated with a nonprincipal ultrafilter on $\bbN$. By \cite{ji2020mip}, not all tracial von Neumann algebras with separable dual are $R^\cU$-embeddable.  

\subsection*{A personal note} Two personal notes are in order. First, at  a dinner in Oberwolfach I pointed out, politely and enthusiastically, that much of~\cite{Kirc:Central} can be construed as model theory. Let’s just say that Eberhard  made it clear (politely) that he did not share my enthusiasm. A couple of years later we collaborated on a model theory paper (\cite{eagle2015quantifier}). The present note does for type~II$_1$ factors what \cite{eagle2015quantifier} did for   \cstar-algebras. 
 Second, I wish that I recorded all  conversations that I had with Eberhard. It took me weeks to process some of the enlightening raw information that he dumped on me, parts of which may be lost for posterity. He will be missed. 

\subsection*{Acknowledgments} I am indebted to the participants of the seminar on Model Theory and Free Probability held at York University in the fall semester 2022, and in particular to Saeed Ghasemi, Pavlos Motakis, and Paul Skoufranis for some insightful remarks. I would also like to thank Srivatsav Kunnawalkam Elayavalli for informing me that he proved Lemma~\ref{L.Nate} independently in February 2023 and for useful remarks on an early draft of this paper, and to 
Isaac Goldbring  for remarks on the first version of this paper and to the referee for a delightfully opinionated (and very useful) report. 

\section{Quantifier elimination in II$_1$ factors}\label{S.QE}

We specialize the definitions from \cite[\S 2.6]{Muenster} to the case of tracial von Neumann algebras. The ongoing discussion applies to any other countable continuous language $\calL$ and, with obvious modifications, to any continuous language $\calL$. For simplicity, we consider only formulas in which all free variables range over the (operator) unit ball. This is not a loss of generality, as an easy rescaling argument shows. 

For $n\geq 0$ and an $n$-tuple of variables $\bar x=\langle x_0,\dots, x_{n-1}\rangle$ (if $n=0$ this is the empty sequence) let $\frFx$ be the $\bbR$-algebra of formulas in the language of tracial von Neumann algebras with free variables included in $\bar x$. For a fixed tracial von Neumann algebra $N$ define a seminorm $\|\cdot\|_N$ on $\frFx$ by ($\varphi^M(\bar a)$ denotes the evaluation of $\varphi(\bar x)$ at $\bar a$ in $M$) 
\[
\|\varphi(\bar x)\|_N=\sup \varphi^{N^\cU}(\bar a)
\]
where $\bar a$ ranges over all $n$-tuples in the unit ball of $N^\cU$. 
(The standard definition takes supremum over all structures $M$ elementarily equivalent to~$N$ and all $n$-tuples in $M$ of the appropriate sort, but by the universality of ultrapowers and the Downward L\"owenheim--Skolem theorem, the seminorms coincide.)

Let $\frFxQF$ denote the $\bbR$-algebra of all quantifier-free formulas in $\frFx$. This is clearly a subalgebra of $\frFx$. 

If $M$ is a subalgebra of $N$ and $\bar a$ in $M$, it is said that $M$ is an \emph{elementary submodel} of $N$ if $\varphi^M(\bar a)=\varphi^N(\bar a)$ for all $\varphi\in \frFx$. An isometric embedding is \emph{elementary} if its range is an elementary submodel. For quantifier-free formulas this equality is automatic, but this is a strong assumption in general. For example, any two elementary embeddings of a separable structure into an  ultrapower\footnote{Ultrapower  of a separable structure associated with a nonprincipal ultrafilter on $\bbN$.} are conjugate by an automorphism of the latter if the Continuum Hypothesis holds\footnote{Continuum Hypothesis has little bearing on the relations between separable structures by Shoenfield’s Absoluteness Theorem; see e.g., \cite{Kana:Book}.} (\cite[Corollary~16.6.5]{Fa:STCstar}).  

\begin{definition}\label{Def.QE}
	The theory of a tracial von Neumann algebra $N$ \emph{admits elimination of quantifiers} if $\frFxQF$ is $\|\cdot\|_N$-dense in $\frFx$ for every $\bar x$. 
\end{definition}

Proposition~\ref{P.QE} below is essentially a special case \cite[Proposition~13.6]{BYBHU} stated with a reference to \cite[pp. 84--91]{hensonultraproducts} in lieu of a proof; see also \cite[Lemma~2.14]{jekel2022free}. Until recently, a self-contained proof of this fact could not be found in the literature. This has finally been remedied in \cite[Proposition~9.4]{hart2023introduction}. For reader's convenience I include a proof using ultrapowers instead of elementary extensions  (see Proposition~\ref{P.QEgeneral}, the equivalence of \eqref{1.QEgeneral}, \eqref{2.QEgeneral}, and \eqref{2+.QEgeneral}).  

\begin{proposition}\label{P.QE} For every tracial von Neumann algebra\footnote{Needless to say, the analogous statement holds for \cstar-algebras and for any other axiomatizable category.} $N$ the following are equivalent. 
	\begin{enumerate}
		\item \label{1.QE} The theory of $N$ admits elimination of quantifiers. 
		\item \label{2.QE} For every finitely generated \wstar-subalgebra $G$ of $N^{\cU}$, every  trace-preserving embedding $\Phi\colon F\to N^\cU$ of a finitely generated \wstar-subalgebra $F$ of $G$ extends to a trace-preserving embedding $\Psi\colon G\to N^\cU$. 

	\begin{tikzpicture}
	\matrix[row sep=1cm,column sep=1cm] {
		\node  (G) {$G$}; &
		\node (NU)  {$N^\cU$};\\ 
		\node (F) {$F$};&
		\node (NNU)  {$N^\cU$};\\
	};
	\draw (G) edge [->] node [above] {$\subseteq$} (NU);
	\draw  (F) edge[->] node [left] {$\subseteq$} (G); 
	\draw  (F) edge[->] node [below] {$\Phi$} (NNU); 
	\draw [dashed] (G) edge[->] node [right] {$\Psi$} (NNU); 
\end{tikzpicture}
	\item \label{2+.QE} Same as \eqref{2.QE}, but for arbitrary separable substructures. 
	\end{enumerate}
\end{proposition}

Clause \eqref{2.QE} resembles the well-known property of the hyperfinite II$_1$ factor~$R$, that any two copies of $R$ in $R^\cU$ are unitarily conjugate (\cite{Jung:Amenability}, \cite{Bro:Topological}),  the analogous (well-known) fact about strongly self-absorbing \cstar-algebras, as well as  the defining property of `Generalized Jung factors' (in this case, automorphisms are not required to be inner, see \cite{atkinson2022factorial}, also see \cite{atkinson2021ultraproduct}) but it is strictly stronger since neither $R$ nor any of the strongly self-absorbing \cstar-algebras admit quantifier elimination. 
The point is that  $F$ and $G$ in \eqref{2.QE}  range over arbitrary finitely-generated substructures of the ultrapower. 

\section{Topological dynamical systems associated to II$_1$ factors}\label{S.Nate}

Lemma~\ref{L.Nate} below is proven by unravelling of the proof of \cite[Corollary~6.11]{Bro:Topological}.   Modulo the standard results (Proposition~\ref{P.QE}, Proposition~\ref{P.QEgeneral}) it implies Theorem~\ref{T.1}. 

\begin{lemma}\label{L.Nate}
	There are a II$_1$ factor $M_1$ with separable predual, a subfactor $M_0$ of $M_1$, and an automorphism $\alpha$ of $M_1$ such that 
		\begin{enumerate}
		\item\label{1.Nate} $\alpha\rs M_0=\id_{M_0}$, but $\alpha\neq \id_{M_1}$, 
	\end{enumerate}
and for every type II$_1$ tracial von Neumann algebra $N$ and every
ultrapower~$N^{\cU}$ the following two conditions hold. 
\begin{enumerate}[resume]
		\item \label{2.Nate} There is a trace-preserving embedding $\Phi\colon M_1\to N^{\cU}$ such that 
		\[
		N^{\cU}\cap \Phi[M_0]'=N^{\cU}\cap \Phi[M_1]'. 
		\]
		\item \label{3.Nate} There is a trace-preserving embedding $\Phi_1$ of the crossed product $M_1\rtimes_\alpha \bbZ$ into $N^\cU$, and $N^{\cU}\cap \Phi_1[M_0]'\neq N^{\cU}\cap \Phi_1[M_1]'$. 
	\end{enumerate}
Moreover, one can choose $M_0=L(\SL(2k+1,\bbZ))$ for any $k\geq 1$ and (regardless on the choice of $M_0$) $M_1=M_0*P$, for any $R^\cU$-embeddable tracial von Neumann algebra $P$ with separable predual $P$.
\end{lemma}

By for example taking $P=L^\infty([0,1])$, we can take $M_0=L(\SL(3,\bbZ)$ and $M_1=L(\SL_3(\bbZ)*\bbZ)$. 

\begin{proof}  Let $\Gamma$ be a property (T) group with infinitely many inequivalent irreducible representations on finite-dimensional Hilbert spaces. 	  For the `moreover' part, fix $k\geq 1$ and an $R^\cU$-embeddable tracial von Neumann algebra $P$ with separable predual, let $\Gamma=\SL(2k+1,\bbZ)$ for any $k\geq 1$, let $M_0=L(\Gamma)$, $M_1=L(\Gamma)*P$, and $\alpha=\id_{M_0}*\alpha_0$, for some nontrivial automorphism $\alpha_0$ of $L(\bbZ)$. Thus \eqref{1.Nate} holds. Since $M_1$ is $R^{\cU}$-embeddable, \eqref{3.Nate} follows by \cite[Proposition~3.4(2)]{anantharaman1995amenable}. 
	
	It remains to prove \eqref{2.Nate}. For $n\geq 2$ let 
	$\rho_n\colon \Gamma\curvearrowright \ell_2(n)$ be an irreducible action of $\Gamma$ on $\ell_2(n)$, if such action exists, and trivial action otherwise. (The choice of $\rho_n$ in the latter case will be completely irrelevant.) Then $\rho_n$ defines a unital *-homomorphism of the group algebra $\bbC\Gamma$ into $M_n(\bbC)$, also denoted~$\rho_n$. Let $\rho=\bigoplus_n \rho_n\colon \bbC\Gamma\to \prod_n M_n(\bbC)$.  
	
	Fix a nonprincipal ultrafilter $\cU$ on $\bbN$ that concentrates on the set $\{n\mid \rho_n$ is irreducible$\}$. If $q\colon \prod_n M_n(\bbC)\to \prod^\cU M_n(\bbC)$ is the quotient map, then $q\circ \rho\colon \bbC\Gamma\to \prod^\cU M_n(\bbC)$ is a unital *-homomorphism.  Let $M_0$ be  the ultraweak closure of  $q\circ \rho[\bbC\Gamma]$. If $\Gamma=\SL(2k+1,\bbZ)$, then 
	\cite[Theorem~1]{bekka2007operator} implies that $M_0$ isomorphic to the group factor $L(\Gamma)$; this nice fact however does not affect the remaining part of our proof.
	 
 For $g\in \Gamma$ we have a representation $(u(g)_n)/\cU$, where $u(g)_n=\rho_n(g)$ is a unitary in $M_n(\bbC)$ for every $n$.
 Since $P$ is $R^\cU$-embeddable, it is also embeddable into $\prod^\cU M_n(\bbC)$. Let $K\subseteq \prod^\cU M_n(\bbC)$ be a countable set that that generates an isomorphic copy of $P$.

 Our copy of $M_0$ in $\prod^\cU M_n(\bbC)$ and the copy of $P$ generated by $K$ need not generate an isomorphic copy of $M_1=M_0*P$. In order to `correct’ this, we invoke some standard results.

	 At this point we need terminology from free probability. A unitary~$u$ in a tracial von Neumann algebra $(M,\tau)$ is a \emph{Haar unitary} if $\tau(u^m)=0$ whenever $m\neq 0$.  Such Haar unitary is \emph{free from} some $X\subseteq M$ if for every $m\geq 1$, if $a_j$ is such that $\tau(a_j)=0$ and in the linear span of $X$ and  if $k(j)\neq 0$ for $0\leq j<m$, then (note that $\tau(u^{k(j)})=0$ since $u$ is a Haar unitary)  
	 \[
	 \tau(a_0 u^{k(0)} a_1 u^{k(1)}\dots a_{m-1} u^{k(m-1)})=0. 
	 \]
	 More generally, if $X$ and $Y$ are subsets of $M$ then  $X$ is \emph{free from} $Y$ if for every $m\geq 1$, if $a_j$  is such that $\tau(a_j)=0$ and in the linear span of $X$ and $b_j$ is such that $\tau(b_j)=0$ and in the linear span of $Y$ for $j<m$, then  
	 		 \[
	 	\tau(a_0 b_0 a_1 b_1\dots a_{m-1} b_{m-1})=0. 
	 	\]
	 	 Since $\Gamma$ is a Kazhdan group (\cite{bekka2008kazhdan}), we can fix a Kazhdan pair $F,\varepsilon$ for $\Gamma$.
	 By \cite[Theorem~2.2]{voiculescu1998strengthened}, there is a Haar unitary $w\in \prod^\cU M_n(\bbC)$  free from $X=\{(u(g)_n)/\cU\mid g\in F\}\cup \{K\}$. 

A routine calculation shows that $K'=\{waw^*\mid a\in K\}$ is free from~$X$ and that $K'$ generates an isomorphic copy of $P$. Therefore the \wstar-subalgebra  of $\prod^\cU M_n(\bbC)$ generated by $\{(u(g)_n)/\cU\mid g\in F\}\cup \{K'\}$ is isomorphic to $M_1\cong M_0*P$. 
	This defines an embedding $\Phi_0\colon M_1\to \prod^\cU M_n(\bbC)$. 
	
	 Fix a type II$_1$ tracial von Neumann algebra $N$. In order to find an embedding of $M_1$ into $N^\cU$ as required in \eqref{2.Nate}, for every $n\geq 1$ write $N=M_n(\bbC)\bar\otimes N^{1/n}$ (where $N^{1/n}$ is the corner of $N$ associated to a projection whose center-valued trace is $1/n$).
Define an embedding of $\prod^\cU M_n(\bbC)$ into $N^\cU$ by (the far right side is $N^\cU$ in disguise)
\[
\textstyle \prod^{\cU} M_n(\bbC)\ni (a_n)/\cU\mapsto (a_n\otimes 1_{N^{1/n}})/\cU\in \prod^{\cU} (M_n(\bbC)\bar\otimes N^{1/n}).
\]
Let $\Phi\colon M_1\to N^\cU$ be the composition of the embedding $\Phi_0$ of $M_1$ into $\prod^\cU M_n(\bbC)$ with this embedding. 

It remains to prove that $N^\cU\cap \Phi[M_0]'=N^\cU\cap \Phi[M_1]'$. 
Since $M_0\subseteq M_1$, only the forward inclusion requires a proof. Suppose that $b\in N^\cU\cap \Phi[M_0]'$ and fix a representing sequence $(b_n)$ for $b$. Fix for a moment $\delta>0$. Then the set  ($F,\varepsilon$ is a Kazhdan pair for $\Gamma$) 
\[
Z_\delta=\{n\mid \max_{g\in F} \|u^g_n b_n -b_n u^g_n\|_2<\varepsilon\delta\}
\]
belongs to $\cU$.\footnote{The set $Z_\delta$ would have belonged to $\cU$ regardless of whether the maximum had been taken over $F$ or some other finite subset of $\Gamma$. The maximum has been taken over $F$ because we need to bound the norm of the commutator of $b$ with all elements of $F$, and not elements of some other finite subset of $\Gamma$.}  

For each $n$ we define an action $\sigma_n\colon \Gamma\curvearrowright L^2(N,\tau)$  as follows. For $f\in \Gamma$ let $\tilde \rho_n(f)=\rho_n(f)\otimes 1_{N^{1/n}}$.  For  $c\in N$, let 
\[
f.c=\tilde\rho_n(f)c\tilde \rho_n(f)^*.
\]
This gives an action by isometries of $\Gamma$ on $(N,\|\cdot\|_{2,\tau})$, which continuously extends to action $\sigma_n$ of $\Gamma$ on $L^2(N,\tau)$. 

For $g\in F$ and $n\in Z_\delta$ we have $\|g.b_n-b_n\|_{2}=\|u^g_n b_n -b_n u^g_n\|_2<\varepsilon\delta$. 
Since $\rho_n$ is an irreducible representation, the space of $\sigma_n$-invariant vectors in $L^2(N,\tau)$ is equal to $1_{M_n(\bbC)}\otimes L^2(N^{1/n},\tau)$. Let $P_n$ be the projection onto this space. 
Since $F,\varepsilon$ is a Kazhdan pair for $\Gamma$, by \cite[Proposition 12.1.6]{BrOz:C*} (with $\Gamma=\Lambda$) or  \cite[Proposition~1.1.9]{bekka2008kazhdan} (the case when~$F$ is compact, and modulo rescaling) we have $\|b_n-P_n(b_n)\|_2<\varepsilon$. 

Therefore $(P_n(b_n))/\cU=(b_n)/\cU$ belongs to $\prod^\cU (1_{M_n(\bbC)}\bar\otimes N^{1/n})$, which is equal to $N^\cU\cap (\prod^{\cU}M_n(\bbC))'$. Recall that $\Phi[M_1]\subseteq \prod^\cU M_n(\bbC)$. Since $b$ in $N^\cU\cap \Phi[M_0]'$ was arbitrary, $N^\cU\cap \Phi[M_0]’$ is included in $N^\cU\cap \Phi[M_1]'$ (and is therefore equal to it), as required. 
\end{proof}

Lemma~\ref{L.Nate} shows, in terminology of \cite[Corollary~6.11]{Bro:Topological} (also \cite{brown2013groups}), that for all type II$_1$ tracial von Neumann algebras  $N$, all $k\geq 1$,  and  all $R^\cU$-embeddable tracial von Neumann algebras with separable predual~$M$, for the natural action of  $\Aut M$ on  $\Hom(L(\SL(2k+1,\bbZ)*M,N^\cU)$ there is an extreme point with trivial stabilizer.

\begin{proof}[Proof of Theorem~\ref{T.1}]  	Suppose for a moment that $N$ is of type II$_1$.  Let $M_0$ and $M_1$ be as in Lemma~\ref{L.Nate}. The embeddings of $M_1$ and $M_1\rtimes_\alpha\bbZ$ into $N^\cU$ provided by \eqref{2.Nate} and \eqref{3.Nate} of this lemma satisfy \eqref{2.QE} of Proposition~\ref{P.QE}, and therefore the theory of $N$ does not admit elimination of quantifiers. 
	
	 In the general case, when $N$ has a type II$_1$ summand, we can write it as $N=N_0\oplus  N_1$, where $N_0$ is type I and $N_1$ is type II$_1$ (e.g., \cite[\S III.1.4.7]{Black:Operator}). Then $N^\cU=N_0^\cU\oplus N_1^\cU$. Let $r=\tau(1_{N_1})$, 
 $P_0=\bbC\oplus M_1$, and $P_1=\bbC\oplus M_1\rtimes_\alpha \bbZ$, with the tracial states $\tau_0$ and $\tau_1$ such that $\tau_0(1_{M_1})=\tau_1(1_{M_1\rtimes_\alpha\bbZ})=r$. Since type II$_1$ algebra cannot be embedded into one of type I, this choice of tracial states forces that every embedding of $M_1$ into $N^\cU$ sends $1_{M_1}$ to the image of~$1_{N_1}$ under the diagonal map. An analogous fact holds for $M_1\rtimes_\alpha\bbZ$. Therefore, as in the factorial case,  Proposition~\ref{P.QE} implies that the theory of $N$ does not admit elimination of quantifiers.
\end{proof}

\section{Proof of Proposition~\ref{P.QE}}
\label{S.Criterion}
As promised, here is a long overdue self-contained proof of  \cite[Proposition~13.6]{BYBHU} (Proposition~\ref{P.QEgeneral} below).  We provide a proof in case when the language $\calL$ is single-sorted and countable. The former is a matter of convenience\footnote{Although the language of tracial von Neumann algebras has a sort for every operator $n$-ball, by homogeneity one can assume that all free variables range over the unit ball anyway.} and the latter requires a minor change of the statement, considering sufficiently saturated models instead of ultrapowers (see \cite[Proposition~13.6]{BYBHU}). 
 
 Let  $\bar a$ be an $n$-tuple of elements in a metric structure $M$ and let $\bar x$ be an $n$-tuple of variables. Each expression of the form ($\varphi(\bar x,y)$ is a formula and~$r\in \bbR$)  
 \[
 \varphi(\bar a,y)=r, \qquad \varphi(\bar a,y)\geq r, \qquad \varphi(\bar a,y)\leq r
 \] 
 is a \emph{condition} (in some contexts called \emph{closed condition}) in $y$ over $\bar a$. A set of conditions over $\bar a$ is a \emph{type} over $\bar a$. If all formulas occurring in conditions of a type are quantifier-free, then the type is said to be quantifier-free. 
Some $c\in N^\cU$ satisfies condition  $\varphi(\bar a,y)=r$ if $\varphi^{N^\cU}(\bar a,c)=r$, and it \emph{realizes the type $\bt(y)$} if it satisfies all of its conditions.  
A type $\bt(y)$ over $\bar a$ in some metric structure $M$ is \emph{consistent} if for every finite set of conditions in $\bt(y)$ and every $\varepsilon>0$ some element of $M$ approximately realizes each of these conditions, up to $\varepsilon$. 

The salient property of ultrapowers (associated with nonprincipal ultrafilters on $\bbN$) is that they are \emph{countably saturated}: every consistent type over a separable set is realized (e.g., \cite{BYBHU}, \cite{Muenster}, \cite{Fa:STCstar}).

The \emph{quantifier-free type of an  $n$-tuple $\bar a$} in $M$ is the set of all conditions of the form $\varphi^M(\bar x)=\varphi^M(\bar a)$, when $\varphi$ ranges over $\frFxQF$.   The \emph{full type of an  $n$-tuple $\bar a$} in $M$ is the set of all conditions of the form $\varphi^M(\bar x)=\varphi^M(\bar a)$, when $\varphi$ ranges over $\frFx$. Quantifier-free and full types are naturally identified with a homomorphism from $\frFxQF$ ($\frFx$, respectively)  into $\bbR$ (see \cite{Muenster}). 
If $X$ is a subset of a metric structure $M$, we may expand the language of $M$ by constant symbols for the element of $X$. The type of $\bar a$ in $M$ in the expanded language is called the \emph{type of $\bar a$ over $X$}.

The following well-known and straightforward lemma is somewhat illuminating. 

\begin{lemma}\label{L.QFtype}
Suppose that $\bar a$ and $\bar a'$ are $n$-tuples in  metric structures in the same language.  Then $a_j\mapsto a'_j$ for $j<n$ defines an isometric isomorphism between the structures generated by $\bar a$ and $\bar a'$ if and only if $\bar a$ and $\bar a'$ have the same quantifier-free type. \qed 
\end{lemma}

While the quantifier-free type of a tuple codes only the isomorphism type of an algebra generated by it,  the full type codes first-order properties such as the existence of square roots of unitaries and Murray--von Neumann equivalence of projections. These existential properties are coded by K-theory, at least in unital \cstar-algebras for which Groethendieck maps are injective. In this case, a unitary $u$ has the $n$-th root if and only if the $K_1$-class $[u]_1$ is divisible by $n$, and projections $p$ and $q$ are Murray--von Neumann equivalent if and only if $[p]_0=[q]_0$ (see \cite{RoLaLa:Introduction} or \cite{blackadar1998k}). However, the information coded by theory and types is different from that coded by K-theory. On the one hand, there are separable AF-algebras with nonisomorphic $K_0$ groups (\cite{debondt2022}) and on the other hand there are separable, nuclear \cstar-algebras with the same Elliott invariant but different theories. All known counterexamples to the Elliott program fall into this category; see  `The theory as an invariant' in \cite[p. 4--5]{Muenster}.

In the following $\equiv$ denotes the relation of elementary equivalence (i.e., sharing the same theory). 

\begin{proposition}\label{P.QEgeneral} If $\calL$ is a countable metric language, then for every $\calL$-structure $N$ the following are equivalent. 
	\begin{enumerate}
		\item \label{1.QEgeneral} The theory of $N$ admits elimination of quantifiers. 
		\item \label{2.QEgeneral} For every finitely generated $\calL$-substructure $G$ of an ul\-tra\-po\-wer\footnote{Again, this ultrapower is associated to a nonprincipal ultrafilter on $\bbN$.}~$N^{\cU}$, every isometric embedding of a finitely generated $\calL$-substructure $F$ of $G$ into $N^\cU$ extends to an isometric embedding of $G$. 
		\item \label{2+.QEgeneral} Same as \eqref{2.QEgeneral}, but for arbitrary separable substructures. 
		\item \label{2-.QEgeneral} For every finitely generated $\calL$-sub\-struc\-tu\-re $G$ of some $M_1\equiv N$, for every isometric embedding of a finitely generated $\calL$-sub\-struc\-tu\-re $F$ of $G$ into $M_2\equiv N$ there is an elementary extension $M_3$ of $M_2$ such that the embedding of $F$ into $M_3$ extends to an embedding of~$G$ into~$M_3$. 
		
		\item \label{3.QEgeneral} If $\varphi(\bar x,y)$ is quantifier free $\calL$-formula, then the formula $\inf_y \varphi(\bar x,y)$ is a $\|\cdot\|_N$-limit of quantifier free formulas. 
	\end{enumerate}
\end{proposition}

\begin{proof} The equivalence of \eqref{2.QEgeneral} and \eqref{2-.QEgeneral} is a well-known consequence of saturation of ultraproducts (see \cite[\S 16]{Fa:STCstar}), and \eqref{2-.QEgeneral} will not be used in this paper (it is included only for completeness).

	For simplicity and without loss of generality we assume that the language~$\calL$ is single-sorted. 
	Assume \eqref{1.QE} and fix finitely generated $\calL$-sub\-struc\-tu\-res $F\leq G\leq N^\cU$ and an isometric embedding $\Phi\colon F\to N^\cU$. By Lemma~\ref{L.QFtype}, $\bar a$ and $\Phi(\bar a)$ have the same quantifier-free type (over the empty set).  It suffices to prove that $\Phi$ extends to an isometric embedding of $G$ in case when~$G$ is generated by $b$ and $F$ for a single element~$b$.

Let $r_\varphi=\varphi(\bar a,b)$ for every $\varphi(\bar x,y)\in \frFxyQF$. 
	Consider the following quantifier-free type over $\Phi(\bar a)$,\footnote{Since the formulas are quantifier-free, we do not need to specify the algebra in which they are being evaluated.} 
	\begin{align*}
		\bt(y)&=\{\varphi(\Phi(\bar a),y)=r_\varphi\mid \varphi(\bar x,y)\in \frFxyQF) \}. 
	\end{align*}
This type is uncountable, but separable in $\|\cdot\|_N$ since the space $\frFx$ is separable in $\|\cdot\|_N$. 

	In order to prove that $\bt(y)$ is satisfiable, fix a finite set of conditions in $\bt(x)$, say $\varphi_j(\bar a,y)=r_j$ for $j<m$, $m\geq 1$. Consider the formula 
	\[
	\psi(\bar x)=\inf_y \max_{j<m} |\varphi_j(\bar x,y)-r_j|. 
	\] 
Then $\psi^{N^\cU}(\bar a)=0$, as witnessed by $b$. Since the theory of $N$ admits elimination of quantifiers, there are quantifier-free formulas $\psi_k(\bar x)$, for $k\geq 1$,  such that $\|\psi(\bar x)-\psi_k(\bar x)\|_N<1/k$ for all $n$. 
Therefore, since $\Phi$ is an isometry between the structures generated by $\bar a$ and $\Phi(\bar a)$, we have  
\[
\psi^{N^\cU}(\Phi(\bar a))=\lim_n \psi_n(\Phi(\bar a))=\lim_n \psi_n(\bar a)=\psi^{N^\cU}(\bar a)=0. 
\]	
Thus $\bt(y)$ is consistent, and by countable saturation (\cite[\S 16]{Fa:STCstar})some $c\in N^\cU$ realizes it. Thus $\bar a,b$ and $\Phi(\bar a),c$ have the same quantifier-free type. By mapping $b$ to $c$ and Lemma~\ref{L.QFtype},  one finds an isometric extension of $\Phi$ to $G$ as required. 

To prove that \eqref{2.QEgeneral} implies \eqref{3.QEgeneral}, assume that \eqref{3.QEgeneral} fails. Fix a quantifier-free formula $\varphi(\bar x,y)$ in $n+1$ variables and $\varepsilon>0$ such that every $\psi\in \frFxQF$  satisfies 
\[
\|\inf_y \varphi(\bar x,y)-\psi(\bar x)\|_N\geq \varepsilon. 
\]
Let $\bt(\bar x(0), \bar x(1))$ be the type in $2n$ variables with the following conditions. 
\begin{align*}
	\psi(\bar x(0))-\psi(\bar x(1))&=0, \text{ for all $\psi\in \frFxQF$, and }\\
	\min(\varepsilon,\inf_y\varphi(\bar x(1),y)-\inf_y\varphi(\bar x(0),y))&\geq \varepsilon.
\end{align*}
	This type is consistent by our assumptions, and by countable saturation it is realized in $N^\cU$ by some $\bar a(0),\bar a(1)$. With $F$ denoting the $\calL$-structure generated by $\bar a(0)$, we have an isometric embedding $\Phi$ of $F$ into $N^\cU$ that sends $\bar a(0)$ to $\bar a(1)$. 
	
	By using countable saturation again, we can find $b\in N^\cU$ such that $\inf_y \varphi(\bar a(1),y)-\varphi(\bar a(0),b)\geq \varepsilon$. Let $G$ be the $\calL$-structure generated by $\bar a(0)$ and $b$. Then $\Phi$ cannot be extended to an isometric embedding of~$G$ into $N^\cU$, and  \eqref{2.QEgeneral} fails. 
	
	Since \eqref{3.QEgeneral} implies that for every quantifier-free formula $\varphi(\bar x,y)$, the formula $\sup_y \varphi(\bar x,y)$ is a uniform $\|\cdot\|_N$-limit of quantifier-free formulas (by replacing~$\varphi$ with $-\varphi$), the proof that \eqref{3.QEgeneral} implies \eqref{1.QEgeneral} follows by induction on complexity of formulas (see \cite{BYBHU}). 
\end{proof}

\section{Concluding remarks}\label{S.Concluding}

The theory of $L^\infty[0,1]$ (with respect to the Lebesgue trace) admits elimination of quantifiers (\cite[Theorem~2.13]{jekel2022free}, see also \cite[Fact 2.10] {benyaacov2013theories} and \cite[Example 4.3]{benyaacov2010continuous}).  
By \cite[Lemma~2.17]{jekel2022free}, every matrix algebra $M_n(\bbC)$ admits elimination of quantifiers. However, if $m\neq n$ then the algebra $M_m(\bbC)\oplus M_n(\bbC)$, with respect to the tracial state $\tau=\frac 12(\tr_m+\tr_n)$ (where $\tr_k$ is the normalized trace on $M_k(\bbC)$), does not admit elimination of quantifiers. This is because the units of the two summands have the same quantifier-free type (the quantifier-free type of a projection is determined by its trace by Lemma~\ref{L.QFtype}), but if $m>n$ then there is no isometric embedding from $M_m(\bbC)$ into $M_m(\bbC)\oplus M_n(\bbC)$ that sends $1_m$ to $1_n$. On the other hand, we have the following. 

\begin{proposition}\label{P.easy} Suppose that $M_m(\bbC)\oplus M_n(\bbC)$ is equipped with a tracial state $\tau$ which has the property that two projections $p$ and $q$ in $M_m(\bbN)\oplus M_n(\bbC)$ are Murray--von Neumann equivalent if and only if $\tau(p)=\tau(q)$.

Then the theory of this algebra admits elimination of quantifiers. 
\end{proposition}

If it is not obvious that tracial states with the property required in Proposition~\ref{P.easy} exist, it may be easier to prove that there are only finitely many tracial states that do not have this property. 
To wit, there are only finitely many Murray-von Neumann equivalence classes of projections in $M_m(\bbC)\oplus M_n(\bbC)$. For any two such distinct classes $[p]$ and $[q]$, there is at most one tracial state $\tau$ such that $\tau(p)=\tau(q)$. (Consider the system of two linear equations in two variables $x$ and~$y$, corresponding to the values of $\tau$ at rank-1 projections of $M_m(\bbC)$ and $M_n(\bbC)$. It has infinitely many solutions if and only if $[p]=[q]$.)

\begin{proof}[Proof of Proposition~\ref{P.easy}]
Suppose that $\tau$ satisfies the assumption,~$F$ is a \cstar-subalgebra of $M_m(\bbC)\oplus M_n(\bbC)$, and $\Phi\colon F\to M_m(\bbC)\oplus M_n(\bbC)$ is a trace-preserving embedding. 
Then $F$ is a direct sum of matrix algebras. Let $p_j$, for $j<k$, be the units of these matrix algebras. Then $\tau(p_j)=\tau(\Phi(p_j))$ for all $j$. By the assumption on $\tau$ there is a partial isometry $v_j$ such that $v_j^* v_j=p_j$ and $v_j v_j^*=\Phi(p_j)$. Therefore $u=\sum_{j<k} v_j$ is a unitary such that $u p_j u^*=\Phi(p_j)$ for $j<k$.  Since every automorphism of a matrix algebra is implemented by a unitary, for every $j<k$ there is $w_j$ such that $w_j^*w_j=w_j w_j^*=p_j$ and for $a\in p_j F$ we have $w_j v_j av_j^* w_j^* =\Phi(a)$. Therefore $\Phi$ coincides with conjugation by the unitary $u'=\sum_{j<k} w_j v_j$. 
 This implies that $\Phi$ automatically extends to an embedding of any $G$ such that $F\subseteq G\subseteq M_m(\bbC)\oplus M_n(\bbC)$ into $M_m(\bbC)\oplus M_n(\bbC)$. 
\end{proof}

The fact that the theory of a fixed structure may or may not admit elimination of quantifiers,  depending on the choice of a language, is hardly surprising. After all, this is exactly what happens with the full matrix algebras, as we pass from the language of tracial von Neumann algebras to the language of \cstar-algebras. 

The idea of the proof of Proposition~\ref{P.easy} should provide the first step towards a confirmation of the following conjecture and a complete answer to \cite[Question~2.18]{jekel2022free}. 

\begin{conjecture}\label{T.2}
	If $T$ is the theory of a tracial von Neumann algebra, then the following are equivalent.  
	\begin{enumerate}
		\item $T$ admits elimination of quantifiers. 
		\item Every model $N$ of $T$ is of type I,  and if $N$ has separable predual then two projections $p$ and $q$ in $N$ are conjugate by a trace preserving automorphism if and only if $\tau(p)=\tau(q)$. 
	\end{enumerate}
\end{conjecture}

Model completeness is a useful weakening of quantifier elimination. 
A theory is \emph{model complete} if every embedding between its models is elementary. As the referee pointed out, the following can be extracted from a well-known semantic characterization of elementary embeddings (see the paragraph between Fact 2.1.2 and Fact 2.1.3 in \cite{AGKE-GeneralizedJung}), but there is some merit in stating it explicitly.

\begin{proposition} \label{P.MC} Assume the Continuum Hypothesis. For a continuous theory $T$ in a complete language the following are equivalent. 
	\begin{enumerate}
		\item \label{1.MC} The theory $T$ is model-complete. 
		\item \label{2.MC} If $M$ and $N$ are separable models of $T$ and $\Phi\colon M\to N$ is an isometric embedding, then there is an isomorphism $\Psi\colon M^\cU\to N^\cU$ such that the following diagram commutes (the horizontal arrows are diagonal embeddings)
		
		\begin{tikzpicture}
			\matrix[row sep=1cm,column sep=1cm] {
				\node  (M) {$M$}; &
				\node (MU)  {$M^\cU$};\\ 
				\node (N) {$N$};&
				\node (NU)  {$N^\cU$};\\
			};
			\draw (M) edge [->] node [left] {$\Phi$} (N);
			\draw  (M) edge[->] (MU); 
			\draw  (N) edge[->] (NU); 
			\draw  (MU) edge[->] node [right] {$\Psi$} (NU); 
		\end{tikzpicture} 
	\end{enumerate} 
\end{proposition}

The use of Continuum Hypothesis in \eqref{2.MC} is innocuous, as it has no effect on projective statements. This is a forcing argument well-known to set theorists but poorly documented in the literature; see for example \cite[Lemma~5.20]{eagle2015quantifier}. It is also a red herring, since in its absence \eqref{2.MC} can be replaced with a considerably more complex, but (once mastered) equally useful, assertion about the existence of a $\sigma$-complete back-and-forth system of partial isomorphisms between separable subalgebras of $M^\cU$ and $N^\cU$ (see \cite[\S 8.2, Theorem~16.6.4]{Fa:STCstar})

\begin{proof} Assume \eqref{1.MC}. Then $\Phi[M]$ is a separable elementary submodel of~	$N^\cU$, and the Continuum Hypothesis  implies that $M^\cU\cong N^\cU$, via an isomorphism~$\Psi$ that sends the diagonal copy of $M$ onto $\Phi[M]$. The \cstar-algebra case is proven in  \cite[Theorem~16.7.5]{Fa:STCstar}, and the proof of the general case is virtually identical.

	Assume \eqref{2.MC} and let $\Phi\colon M\to N$ be an embedding between models of~$T$. We need to prove that $\Phi$ is elementary. By replacing $M$ and~$N$ with separable elementary submodels large enough to detect a given failure of elementarity, we may assume they are separable.  With $\Psi$ as guaranteed by \eqref{2.MC}, the embedding of $M$ into $N^\cU$ is, being the composition of an elementary embedding and an isomorphism, elementary. Since the diagonal image of $N$ in $N^\cU$ is elementary, it follows that $\Phi$ is elementary.  
\end{proof}

\begin{conjecture} The theory of a tracial von Neumann algebra $N$ is model complete if and only if $N$ is of type I. 
	\end{conjecture}

A fact relevant to both conjectures is that by \cite[Proposition~3.1]{farah2023preservation}, elementarily equivalent type I tracial von Neumann algebras with separable preduals are isometrically isomorphic. 	
Following a referee's advice, we state the following one-time side remark in the theorem environment, with a proof.

\begin{proposition}
	The theory of every type I tracial von Neumann algebra  is model complete. 
\end{proposition}

\begin{proof}
	This is an immediate consequence of the results of \cite{farah2023preservation}. 	Fix a type I tracial von Neumann algebra $M$. 
	
	If $M$ is abelian, then $M\cong L^\infty(\mu)$ for a probability measure $\mu$. Then \cite[Lemma~3.2]{farah2023preservation}  implies that theory of $M$ determines the measures of the atoms of $\mu$ (these are the values of $\rho(1,n)$ for $n\geq 1$, using the notation from this lemma and taken in the decreasing order; $\rho(1,0)$ is the measure of the diffuse part). 
	
	Therefore if $N\equiv M$ then $N\cong L^\infty(\nu)$, and the measures of the atoms of $\nu$ (if any) are exactly the same as the measures of the atoms of $\mu$. Every trace-preserving *-homomorphism $\Phi$ from $M$ into $N$ has to send the diffuse part to the diffuse part. Since these parts are of the same measure, it also sends atomic part to the atomic part. Finally, since the atoms in $M$ and $N$ have the same measures (with multiplicities), and since we are dealing with a probability measure, the restriction of $\Phi$ to the atomic part is an isomorphism.  The theory of $L^\infty$ space  of a diffuse measure admits elimination of quantifiers, hence the restriction of $\Phi$ to the diffuse part is elementary. Thus $\Phi$ can be decomposed as a direct sum of two elementary embeddings, and is therefore elementary by the second part of   \cite[Corollary~2.6]{farah2023preservation}. 
	
	This proves Proposition in case when $M$ is abelian. 
	
	In general, when $M$ is not necessarily abelian, it is isomorphic to $\bigoplus_{n\in X} A^M_n$, where $X\subseteq [1,\infty)$, and for every $n\in X$ we have $A^M_n\cong M_n(B^M_n)$, for a commutative von Neumann algebra $B^M_n$. 
	Suppose that $M\equiv N$. Then 
	\cite[Lemma~3.2]{farah2023preservation}  implies that $N\cong \bigoplus_{n\in X} A^N_n$ (with the same set $X$) and that $\tau(1_{A^M_n})=\tau(1_{A^N_n})$ for all $n\in X$. (Using the notation from 	\cite[Lemma~3.2]{farah2023preservation}, these quantities are both equal to $\sum_k \rho_M(n,k)$, for every $n\in X$.)

Now assume that $\Phi\colon M\to N$ is a trace-preserving *-homomorphism. For $n'<n$ in $X$ we have that $\Phi(1_{A^M_n})$  	and $1_{A^N_{n'}}$ are orthogonal. By induction on $n\in X$, this implies that $\Phi(1_{A^M_n})=1_{A^N_{n}}$ for all $n$. 

By the abelian case,  the restriction of $\Phi$ to $A^M_n$ is elementary  for every~$n$. By  the second part of   \cite[Corollary~2.6]{farah2023preservation}, $\Phi$ is elementary. 
\end{proof}

The original version of this paper contained an appendix by Srivatsav Kunnawalkam Elayavalli, removed following a suggestion of the referee. The appendix was concerned with the two notions that seem to be sensitive to the choice of additional axioms of ZFC, and the author's strong opinion on such definitions is expressed in a long footnote on p.~12 of \cite{farah2023quantifierv2}.

\bibliographystyle{plain}
\bibliography{qe}

\end{document}